\documentclass{patmorin}
\usepackage{pat} 
\usepackage[utf8x]{inputenc}
\usepackage[T1]{fontenc}
\usepackage{skull}
\usepackage{graphicx}
\usepackage{float}
\usepackage{subfig}
\usepackage{amssymb}
\usepackage{thmtools, thm-restate}

\newcommand{\tbcomment}[1]{{\color{cyan!50!black}TB says: #1}}

\Crefname{figure}{Fig.}{Figs.}
\Crefname{prop}{Prop.}{Props.}
\Crefname{thm}{Thm.}{Thms.}

\allowdisplaybreaks

\usepackage[normalem]{ulem}
\usepackage{cancel}
\usepackage{enumitem}
\usepackage[textsize=scriptsize]{todonotes}
\usepackage{pifont}

\usepackage[longnamesfirst,numbers,sort&compress]{natbib}

\usepackage[mathlines]{lineno}
\newcommand*\patchAmsMathEnvironmentForLineno[1]{%
 \expandafter\let\csname old#1\expandafter\endcsname\csname #1\endcsname
 \expandafter\let\csname oldend#1\expandafter\endcsname\csname end#1\endcsname
 \renewenvironment{#1}%
    {\linenomath\csname old#1\endcsname}%
    {\csname oldend#1\endcsname\endlinenomath}}%
\newcommand*\patchBothAmsMathEnvironmentsForLineno[1]{%
 \patchAmsMathEnvironmentForLineno{#1}%
 \patchAmsMathEnvironmentForLineno{#1*}}%
\AtBeginDocument{%
\patchBothAmsMathEnvironmentsForLineno{equation}%
\patchBothAmsMathEnvironmentsForLineno{align}%
\patchBothAmsMathEnvironmentsForLineno{flalign}%
\patchBothAmsMathEnvironmentsForLineno{alignat}%
\patchBothAmsMathEnvironmentsForLineno{gather}%
\patchBothAmsMathEnvironmentsForLineno{multline}%
}

\newenvironment{txteq*}
  {
    \begin{equation*}
    \begin{minipage}[c]{0.85\textwidth} 
    \em                                
  }
  {\end{minipage}\end{equation*}\ignorespacesafterend}

\newcommand{\spa}{\mathsf{span}}
\newcommand{\sk}{\mathsf{sk}}
\newcommand{\ch}{\mathsf{ch}}

\newcommand{\cmark}{\ding{51}}%
\newcommand{\xmark}{\ding{55}}%
\newcommand{\sm}{\mathbin{\setminus}}

\newcommand{\facei}{\mathsf{FC}}
\usepackage[dvipsnames,hyperref]{xcolor}
\usepackage{hyperref}
\definecolor{linkblue}{named}{MidnightBlue}
\hypersetup{colorlinks=true, linkcolor=Blue,  anchorcolor=Blue,
	citecolor=Blue, filecolor=Blue, menucolor=Blue,
	urlcolor=Blue,linkcolor={Blue},
urlcolor={Blue}}

\newcommand\subsetcong{\mathrel{\text{%
    \setbox0\hbox{$\subseteq$}%
    \rlap{\hbox to \wd0{\hss\hss\hss\raisebox{1.5\height}{$\sim$}\hss}}\box0
}}}
\makeatother

\renewcommand{\le}{\leqslant}
\renewcommand{\leq}{\leqslant}

\renewcommand{\geq}{\geqslant}

\title{\MakeUppercase{The basis number OF 1-PLANAR GRAPHS}}

\author{%
Saman Bazargani \and\,
Therese Biedl
  \thanks{David R.~Cheriton School of Computer Science, University of Waterloo, Waterloo, Canada.   
  {\tt biedl@uwaterloo.ca}. Research supported by NSERC.}\,\,\thanks{Corresponding author}\,
\and\quad
 Prosenjit 
 Bose\thanks{School of Computer Science, Carleton University, Ottawa, Canada.  
 Emails: {\tt jit@scs.carleton.ca}, {\tt anil@scs.carleton.ca}, and {\tt bobby.miraftab@gmail.com}.
 Research supported by NSERC. 
 }
\and\quad
Anil Maheshwari\footnotemark[3]
\and\quad
Babak Miraftab \footnotemark[3]\,\,\footnotemark[2]\
}

\date{}

\begin{document}
\maketitle


\begin{abstract}
Let $\mathcal{B}$ be a set of Eulerian subgraphs of a graph $G$. We say $\mathcal{B}$ forms a \defin{$k$-basis} if it is a minimum set that generates the cycle space of $G$, and any edge of $G$ lies in at most $k$ members of $\mathcal{B}$. The \defin{basis number} of a graph $G$, denoted by $b(G)$, is the smallest integer such that $G$ has a $k$-basis. A graph is called \defin{1-planar} (resp. \defin{planar}) if it can be embedded in the plane with at most one crossing (resp. no crossing) per edge.
MacLane's planarity criterion characterizes planar graphs based on their cycle space, stating that a graph is planar if and only if it has a $2$-basis. We study here the basis number of 1-planar graphs, demonstrate that it is unbounded in general, and show that it is bounded for many subclasses of 1-planar graphs.
\end{abstract}

\section{Introduction}

For any graph $G=(V,E)$, the cycle space $\mathcal{C}(G)$ of $G$ consists of all Eulerian subgraphs of $G$.    (Detailed definitions are in \Cref{sec:prelim}.)   A basis of the cycle space consists of a minimum set of Eulerian subgraphs from which we can generate all elements of $\mathcal{C}(G)$ via taking symmetric differences (or equivalently, linear combinations over $\mathbb{F}_2$).
Extensive research has been conducted in the field of cycle basis theory, and a multitude of results have been 
obtained
by various researchers \cite{polytime,shortest,minimumcycles}. 
Notably, it is regarded as a source of matroid theory \cite{homotop, whitney}, a branch of mathematics that focuses on linear dependency. 
The practical applications of cycle basis theory extend to diverse fields, including electric circuit theory \cite{chua1973optimally}, and chemical structure storage and retrieval systems \cite{downs1989review}. 

\noindent A \defin{planar graph} is a  graph that can be drawn in a plane without any of its edges crossing.
A result by MacLane
\cite{maclane1970combinatorial} characterizes planar graphs in terms of their cycle space. 

\begin{lem}[MacLane’s planarity criterion \cite{maclane1970combinatorial}] \label{maclane}A graph is planar if and only if there is a basis for its cycle space such that every edge lies in at most two elements of the basis.
\end{lem}

Schmeichel \cite{MR615307} introduced a specific term for the property described in MacLane’s planarity criterion:
Call a basis $\mathcal{B}$ of the cycle space a \defin{$k$-basis} if every edge belongs to at most $k$ members of $\mathcal{B}$, and let the \defin{basis number} of $G$ be the smallest $k$ such that $G$ has a $k$-basis.
Then MacLane's planarity criterion can be phrased equivalently that planar graphs are exactly the graphs with basis number at most 2.

Numerous authors have investigated the basis number of other kinds of graphs, and how it changes when doing graph operations. 
Schmeichel \cite{MR615307} showed that the basis number can be unbounded, and
gave bounds on the basis number of complete graphs, hypercubes and some complete bipartite graphs. 
Banks and Schmeichel \cite{MR685059} and later McCall \cite{MR801601} studied these graph classes further and made the bounds tight. 
The basis number of specific graphs, such as the Petersen graph, the Heawood graph, and some cage graphs, was discussed by Alsardary and Ali \cite{MR1983447}.
As for graph operations, 
Alsardary \cite{MR1844307} and
Alzoubi \cite{cartesian_basis_1,cartesian_basis_2} 
studied Cartesian products of certain graphs of known basis number, 
Alzoubi and  Jaradat \cite{lexico_basis} 
studied lexicographic products,
and Jaradat \cite{strong_basis} considered the strong product with a bipartite graph.

In this paper, motivated by MacLane's planarity criterion, we ask what the basis number is for graphs that are 
not planar,  but that are ``close to planar'' in some sense.   To our knowledge, there were no results in this area until now.
In very recent research (concurrent with our own paper)  it was 
proved
that every 
graph with genus 1, and in particular, every 
toroidal graph, has a 3-basis  \cite{lehner2024basis}.


We investigate here graphs 
that generalize planarity in a different way, and specifically, 
study the basis number of~$1$-planar graphs.
A graph~$G$ is \defin{$1$-planar}, if it can be drawn such that each edge is crossed at most once.    
It would be a natural guess to think that they also have small basis number.    But as our first main result, we show that this is \emph{not} the case:
for every integer $\ell$, there exists a $1$-planar graph $G$ with 
basis number $\ell$ or more,
see \Cref{notbounded}.   In fact, this result even holds under further restrictions on the graph, such as bounding the maximum degree, or forbidding two crossings to share vertices.

We then turn towards subclasses of 1-planar graphs that \emph{do} have a bounded basis number.    We first show that if (in some 1-planar drawing) the set of edges without crossings forms a connected subgraph, then the basis number is at most 4.  With this, we immediately get this bound on the basis number for some frequently studied subclasses such as locally maximal 1-planar graphs, which include the optimal 1-planar graphs (definitions of these will be given below).   Pushing the results further, we show that for some subclasses (usually obtained by restricting configurations near crossings further) the basis number is at most 3.     

Many of our arguments rely on showing that local graph operations, such as contracting and adding an edge, do not change the basis number significantly.    
As this may be of interest on its own, we group all these results together in \Cref{sec:operations}, before giving the results for 1-planar graphs in \Cref{sec:1planar} and concluding with open problems in \Cref{sec:conclusion}.

\section{Preliminaries} 
\label{sec:prelim}


We refer the readers to \cite{diestel} for notation and terminology of the graph theoretical terms and to \cite{dimension} for linear algebra notation,
and provide here brief refreshers of only the most important (or non-standard) definition.    
Throughout, let $G=(V,E)$ be a graph with $n$ vertices and $m$ edges.   

\paragraph{Cycle space.}    
The \defin{cycle space} of a graph~$G$, denoted by~$\mathcal{C}(G)$, is the 
collection of all subgraphs of $G$ that have the form $C_1+\dots+C_k$ where $C_1,\dots,C_k$ are cycles of $G$ and $+$ denotes the symmetric difference.
Equivalently, the cycle space of a graph is defined as the set of all \defin{Eulerian subgraphs} of $G$, i.e., subgraphs where all vertex-degrees are even.   (The empty subgraph is considered to be Eulerian, and is denoted by $\emptyset$.)
The set
$\mathcal{C}(G)$ can be viewed as a vector space 
over the finite field $\mathbb{F}_2$, i.e., addition means taking the symmetric difference, and scaling is permitted only with 0 or 1.   
We can therefore use standard vector space terminology, such as \defin{generating set} and \defin{basis}, also for the cycle space $\mathcal{C}(G)$.   To remind the reader of these terms:   We say that a set
$S=\{v_1,\dots,v_s\}$ of vectors in a vector space $V$ \defin{generates 
$V$} if every vector in $V$ is a linear combination of $v_1,v_2,\dots,v_s$.
So for example the set of all cycles of a graph $G$ is a generating set of $\mathcal{C}(G)$.
A generating set~$S$ is said to be \defin{minimal} if 
no proper subset $S'$ of $S$ generates $V$, or equivalently, 
no vector in~$S$ can be expressed as a linear combination of the other vectors in~$S$. 
A \defin{basis} for a vector space $V$ is a minimal generating set of~$V$.
Any basis has the same cardinality, which we call the \defin{dimension} of $V$ and denote by $\dim(V)$.
It is 
clear that any set $S$ that generates $V$ contains a subset that is a basis.

A natural question is how to find a basis for the cycle space $\mathcal{C}(G)$ of a graph $G$.
The following lemma gives one possible approach.

\begin{lem}{\rm\cite[Theorem 1.9.5]{diestel}}
Let~$G=(V,E)$ be a connected graph and let $T=(V,E')$ be a spanning tree of $G$. For each edge~$e\in E\sm E'$ let the \defin{fundamental cycle of $e$} 
be the unique cycle within $T+e$. 
Then the set of fundamental cycles with respect to~$T$ forms a basis of ~$\mathcal C(G)$. 
\end{lem}

In particular, the dimension of the cycle space is always
$\beta(G)=|E|-|V|+1$. This number is known in algebraic
topology as the \defin{Betti number} of~$G$, see \cite{MR1280460}.
Indeed, 
the cycle space of a graph naturally appears in the field of algebraic topology (where it is known as the first homology group ~$\mathsf{H}_1(G,\Z_2)$ of the graph), but we will not pursue this line of thinking further.

\paragraph{Basis number.}
In this paper, we are interested in not only finding a basis, but in finding one that ``distributes the load'' fairly, i.e., every edge is not used too often by elements of the basis.

\begin{defn}
Let~$G$ be a graph and let $\mathcal{B}$ be a collection of subgraphs of $G$.
For any edge $e$, define the \defin{charge} $ch_{\mathcal{B}}(e)$ (with respect to $\mathcal{B}$) be the number of elements of $\mathcal{B}$ that contain $e$.
We call~$\mathcal{B}$ a \defin{$k$-basis} of $G$ if it is a basis of $\mathcal{C}(G)$ and
if $ch_{\mathcal{B}}(e)\leq k$ for all edges.   
\end{defn}

In what follows, we will often drop the subscript $\mathcal{B}$ from the charge $ch(e)$ since the collection $\mathcal{B}$ will be clear from context.   Also recall that the  \defin{basis number} of $G$, denoted $b(G)$, is the smallest $k$ such that $G$ has a $k$-basis.

\begin{exa}
We illustrate these concepts with the graph $G=K_{3,4}$ in
\Cref{babyheawood}.   One can easily verify that the set $\mathcal{B}$ of cycles in \Cref{babyheawood}(a) gives 
charge 3 or less to each edge.   To verify that it is a basis, first observe that it has the correct cardinality, i.e., uses $|E|-|V|+1$ cycles.   Then we must verify that it generates the cycles space.  One possible method for this is to pick a spanning tree $T$, and to verify that all fundamental cycles of $T$ are generated with $\mathcal{B}$; since these fundamental cycles themselves generate $\mathcal{C}(G)$, so does $\mathcal{B}$.   For example, for the tree $T$ in \Cref{babyheawood}(b), one can verify that fundamental cycle $T+ 45$ can be generated as the symmetric difference of cycles 367452, 234567, 1436, 4127 and 4521 in $\mathcal{B}$.   All other fundamental cycles of $T$ can likewise be generated and so this is a 3-basis.
\begin{figure}[H]
\centering 
\subfloat[A set $\mathcal B$ of cycles.]
{{\includegraphics[scale=1]{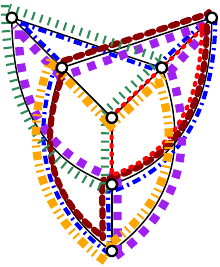}}}
\qquad
\qquad
\subfloat[A spanning tree (gray).]
{{\includegraphics[scale=1]{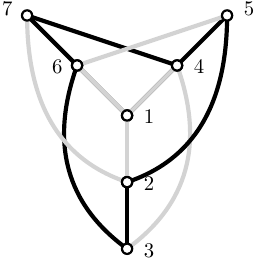}}}
\caption{A $1$-plane graph with a 3-basis.}
\label{babyheawood}
\end{figure}

\end{exa}

\paragraph{Connectivity.}   A graph $G$ is called \defin{connected} if there exists a path between any two
vertices of $G$.   A maximal connected subgraph is called a \defin{connected component}. A \defin{cutting set} of $G$ is a set $S$ of vertices such that $G\setminus S$ has more connected components of $G$.   If $|S|=1$ then its unique vertex is called a \defin{cutvertex}.   A graph is \defin{2-connected} if it has no cutvertices.
A \defin{block} is a maximal 2-connected subgraph. 
Any cycle of $G$ must reside within one block.    Therefore any Eulerian subgraph (which can be partitioned into cycles) can be generated if we have a basis in each block.   Since blocks are edge-disjoint, therefore only the basis number of each block is relevant:

\begin{fact}
Any graph $G$ with blocks $B_1,\dots,B_t$ has basis number
$b(G)=\max_{1\leq j\leq t} b(B_j)$.
\end{fact}

In particular, we only need to consider 2-connected graphs when computing the maximum possible basis number of graphs in a graph class.

\paragraph{Planar and 1-planar graphs.}

A \defin{planar graph} is a graph that can be drawn in the plane in such a way that its edges intersect only at their endpoints. Such a drawing is called a \defin{plane graph}.
For a plane graph $G$,
the regions of~$\mathbb R^2\sm G$ are called the \defin{faces} of~$G$. 
The unbounded face is called the \defin{outer-face} of~$G$. 
Every face $f$ naturally gives rise to a subgraph of $G$ (the \defin{boundary of $f$}) obtained by taking all edges that lie on the closure of $f$.   A planar graph is 2-connected if and only if all face-boundaries are cycles.   With this, a 2-basis of a planar graph is easy to find.

\begin{fact}\label{facial_cycles}
Let~$G$ be a~$2$-connected  plane graph and let $\mathcal{B}$ be the set of all boundaries of all faces except the outer-face.   Then $\mathcal{B}$ is a 2-basis and every outer-face edge has charge 1.

\end{fact}

A graph~$G$ is \defin{$1$-planar} if it can be drawn in the plane such that each edge is crossed at most once.   Here, ``drawn'' permits edges to cross each other, but only at points that are interior to both, and no three edges cross at a point.   Also, no two edges with a common endpoint are allowed to 
cross each other.   If a~$1$-planar graph is drawn this way, the drawing is called a~\defin{$1$-plane graph}.
These graphs are the main class of interest in this paper, and we will give further definitions related to them in \Cref{sec:1planar}.

For both planar and 1-planar graphs, we normally consider drawings in the plane, but occasionally will think of them as drawings on the sphere instead; the only difference is then that there is no longer an unbounded region, hence no outer-face.

\section{Operations and Basis Numbers}
\label{sec:operations}

In this section, we explore a number of graph operations, and how they affect the basis number of a graph; some of these give us useful tools to analyze basis numbers later.
We start by exploring contracting or adding an edge.    Formally, for an edge $e=uv$ in a graph $G$, the graph $G/e$ \emph{obtained by contracting} $e$ is the graph obtained by deleting edge $e$ and combining $u$ and $v$ into one new vertex $x$ that inherits all edges other than $e$ that were incident to $u$ or $v$.    (In particular, if there is a parallel edge to $e$ then $G/e$ has a loop at $x$, and if the pair $u,v$ has a common neighbour $y$ then $G/e$ has two parallel edges $xy$.)
Also, for any vertex pair $u,v$, we write $G+uv$ for the graph obtained by adding an edge $uv$ (if edge $uv$ already existed in $G$, then we add a parallel copy).

\begin{lem}\label{contract}
For any connected graph $G$ the following holds:
\begin{enumerate}
\item $b(G/e)\leq b(G)$ for any edge $e$ of $G$,  and
\item $b(G+uv)\leq b(G)+1$ for any vertex pair $u,v$.
\end{enumerate}
\end{lem}

\begin{proof}
Let $\mathcal B=\{C_1,\ldots,C_k\}$ be a $b(G)$-basis of $G$.
To prove (1), let  $\mathcal B'=\{C_1',\ldots,C_k'\}$ be the set obtained from $\mathcal B$ by contracting $e$ in all Eulerian subgraphs that contain it;
clearly every edge has charge at most $b(G)$ in $\mathcal{B}'$. 
We want to show that $\mathcal B'$ is a basis of $G/e$.
Since $\beta(G)=\beta(G/e)$ the two cycle spaces have the same dimension.   By $|\mathcal{B}|=|\mathcal{B}'|$ it is therefore enough to show that $\mathcal B'$ is minimal, i.e., no element of $\mathcal B'$ can be generated using the others.
Assume to the contrary that there exists a set of Eulerian subgraphs $C_{i_1}',\ldots,C_{i_j}'\in \mathcal B'$ such that 
$C_{i_1}'+\cdots+C_{i_{j-1}}'=C_{i_j}'$
or equivalently
$C_{i_1}'+\cdots+C_{i_j}'= \emptyset$. 
Considering the corresponding Eulerian subgraphs $C_{i_1},\dots,C_{i_j}$ in $\mathcal{B}$, 
we note that $C_{i_1}+\cdots+C_{i_j}\neq \emptyset$ since $\mathcal B$ is minimal.
So $C_{i_1}+\cdots+C_{i_j}=\{e\}$, but this is impossible
since $\{e\}$ is not an Eulerian subgraph.

To prove (2), let
$u,v$ be a vertex pair, and define $G'\coloneqq G+uv$. 
Since $G$ is connected, $G'$ contains a cycle $C$ that goes through $uv$. Pick an arbitrary such cycle and define $\mathcal{B}':= \mathcal{B} \cup \{C\}$;   
clearly every edge has charge at most $b(G)+1$ in $\mathcal{B}'$. 
We will show that $\mathcal{B}'$ is a basis of $G'$, for which by $\beta(G')=\beta(G)+1$ it suffices to show that it is minimal.
Assume for contradiction that 
$C_{i_0}+C_{i_1}+\cdots+C_{i_j}=\emptyset$ for some $C_{i_j}\in \mathcal{B}'$ and $j\geq 1$. 
Since basis $\mathcal{B}$ is minimal, this set must include $C$, say $C_{i_0}=C$.
Since $C$ includes the newly added edge $uv$, but none of 
$C_{i_1},\dots,C_{i_j}$ does, this is impossible.
\end{proof}

In conjunction with \Cref{maclane},
this gives us an immediate result for graphs that have \defin{skewness} $\ell$, i.e., graphs that can be obtained from a planar graph by adding $\ell$ arbitrary edges.
\begin{cor}
\label{cor:skewness}
Any graph with skewness $\ell$ has a basis number at most $2+\ell$.
\end{cor}

We next turn towards the operation of subdividing an edge $e=uv$
, i.e., deleting the edge and adding a new vertex $w$ adjacent to $u$ and $v$.    We will actually first define and analyze a more complicated operation that generalizes subdivision.
Let $G$ be a graph with an edge $e$ and $H$ be a connected graph with two vertices (\defin{terminals}) $s$ and $t$.   To \defin{replace edge $e$ by $H$} means to take the union of $G$ and $H$, delete edge $e$, and then identify the ends of $e$ with the terminals of $H$.   See also \Cref{blow_up}. Note in particular the operation of subdividing edge $e$ can be viewed as replacing $e$ by a path with two edges. 
\begin{figure}[H]
\centering
{{\includegraphics[scale=0.8]{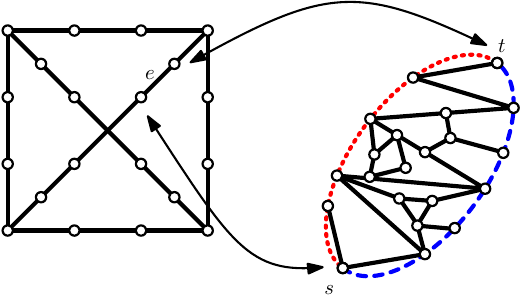}}}
\qquad
\qquad
{{\includegraphics[scale=0.6]{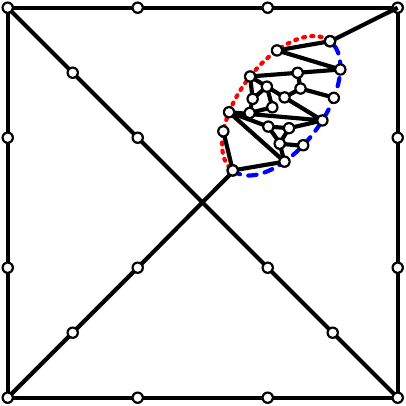}}}
\caption{Replacing edge $e$ of $G$ by a 2-connected planar graph $H$ with terminals $s,t$.  The paths in $\mathcal{P}$ are red (dotted) and blue (dashed).}
\label{blow_up}
\end{figure}

It would not be hard to show that if some graph $M$ results from such a replacement, then $b(M)\leq b(G)+b(H)$.   We can strengthen this result by taking into account that $H$ may have a basis with ``room to add more paths from $s$ to $t$''. Formally, for a graph $H$ with terminals $s,t$, and any integer $\ell$, we define the \defin{$\ell$-augmented basis number} to be the minimum $k$ such that $H$ has a basis $\mathcal{B}_H$ and a set of $\ell$ (not neccessarily disjoint) $st$-paths $\mathcal{P}=\{P_1,\dots,P_\ell\}$ such that every edge of $H$ has charge at most $k$ with respect to $\mathcal{B}_H\cup \mathcal{P}$.   Obviously $k\leq b(H)+\ell$, but often it is better. For example, if $H$ is a planar 2-connected graph with $s,t$ on the outer-face, then its 2-augmented basis number is 2, because we can use as $\mathcal{B}_H$ the basis from \Cref{facial_cycles} (which uses outer-face edges only once), and use as $\mathcal{P}$ the two paths connecting $s$ and $t$ along the outer-face.
\begin{prop} \label{blowup}
Let $G$ be a graph with edge $e$, and
$M$ be the graph obtained by replacing $e$ by a connected graph $H$ with terminals $s,t$.  For some $b(G)$-basis $\mathcal{B}_G$ of $G$, let $\ell\leq b(G)$ be the charge of $e$ in $\mathcal{B}_G$, and let $k\leq b(H)+\ell$ be the $\ell$-augmented basis number of $H$.   Then $b(M)\leq \max\{b(G),k\}$.
\end{prop}
\begin{proof}
Let $\mathcal{B}_e=\{B_1,\dots,B_\ell\}$ be the elements of $\mathcal{B}_G$ that use edge $e$.
Fix a basis $\mathcal{B}_H$ of $H$ and $\ell$ paths $\mathcal{P}=\{P_1,\dots,P_\ell\}$ that achieve the $\ell$-augmented basis number of $H$.   Define
\[
    \mathcal{B}_M \quad = \quad \mathcal{B}_G \setminus \mathcal{B}_e \quad \cup \quad \mathcal{B}_H \quad \cup \quad \sum_{i=1}^\ell (B_i - e + P_i),
\]
i.e., take the union of the 
bases,
except replace any use of edge $e$ with one of the paths through $H$.   Observe that (with respect to $\mathcal{B}_M$), all edges of $G\setminus e$ have the same charge as in $\mathcal{B}_G$, and all edges in $H$ have the same charge as in $\mathcal{B}_H\cup \mathcal{P}$, so all edges have charge at most $\max\{b(G),k\}$.   It hence only remains to show that $\mathcal{B}_M$ generates any element $C$ in $\mathcal{C}(M)$; we can then find a suitable basis within $\mathcal{B}_M$.       We may assume that $C$ is a simple cycle (since all Eulerian subgraphs can be generated from cycles)   and have three cases for where $C$ could reside.

\medskip\noindent{\bf Case 1:} $C$ uses only edges of $H$.    Then $C$ can be generated with elements of $\mathcal{B}_H$.

\medskip\noindent{\bf Case 2:} $C$ uses only edges of $G\setminus e$.    Then $C$ can be generated with elements of $\mathcal{B}_G$.   Since $e\not\in C$, this linear combination must use an even number (possibly 0) of elements of $\mathcal{B}_e$, say $C=\sum_{j=1}^{2t} B_{i_j} + C'$ where $B_{i_j}\in \mathcal{B}_e$ while $C'$ can be generated by $\mathcal{B}_G\setminus \mathcal{B}_e$.   Now write
\[
    C = \sum_{j=1}^{2t} (B_{i_j} - e + P_{i_j}) \quad  + \quad C' \quad + \quad \sum_{j=1}^{2t} P_{i_j} \quad + \quad \sum_{j=1}^{2t} e
\]
(recall that $1\leq i_j\leq \ell$ and therefore there exists an $st$-path $P_{i_j}\in \mathcal{P}$ with the same index).
The term $\sum_{j=1}^{2t} e$ is actually 0 since we add over $\mathbb{F}_2$.
Since $\sum_{j=1}^{2t} P_{i_j}$ visits $s$ and $t$ an even number of items, it is an Eulerian subgraph of $H$ and so can be generated with $\mathcal{B}_H$.    All other items in the sum are in $\mathcal{B}_M$ or can be generated from $\mathcal{B}_G\setminus \mathcal{B}_e$, so $C$ can be generated with $\mathcal{B}_M$.

\medskip\noindent{\bf Case 3:} $C$ uses edges in both $G\setminus e$ and $H$.    Then $C$ must visit vertices that were combined, i.e., $s$ and $t$.  Since $C$ is a simple cycle, it visits each of them exactly once, so we can write $C=C_G + C_H$ where $C_G$ and $C_H$ are $st$-paths in $G\setminus e$ and $H$.   Consider the two subgraphs $C_G+B_1-e$ and $C_H+P_1$ which can readily be observed to be Eulerian subgraphs.    They also reside in $G\setminus  e$ and $H$, respectively, and by the previous cases can therefore be generated with $\mathcal{B}_M$.   Since
\[
    C \quad =  \quad C_G+C_H  \quad =  \quad (C_G+B_1-e) \quad + \quad (C_H+P_1) \quad + \quad (B_1-e + P_1)
\]
and $B_1-e+P_1$ is in $\mathcal{B}_M$, we can therefore generate $C$ as well.
\end{proof}

To see a specific example, recall that a planar 2-connected graph with two terminals on the outer-face has 2-augmented basis number 2.    Therefore we can replace edges by such graphs without increasing the basis number.

\begin{cor}
Let $G$ be a graph with an edge $e$, and let $G'$ be obtained by replacing $e$ with a 2-connected planar graph $H$ with two terminals on the outer-face.
Then $b(G')=\max\{b(G),2\}$.
\end{cor}


In particular, we can choose $H$ to be a pair of parallel edges $(s,t)$.   Clearly $b(H)=1$ and the $\ell$-augmented basis number of $H$ is $1+\lceil \ell/2 \rceil$.
Since replacing $e$ by $H$ is the same as duplicating $e$, and $1+\lceil b/2 \rceil \leq b$ for $b\geq 2$, we get the following strengthening of \Cref{contract}(2) for a special case.

\begin{cor} 
Let $G$ be a graph with an edge $e$, and let $G'$ be the graph obtained by adding a second copy of $e$.   Then $b(G')\leq \max\{b(G), 2\}$.
\end{cor}

Finally, we prove the desired result for subdividing edges, which is even stronger than previous results in that the basis number does not change at all.

\begin{lem}\label{subdivide}
\label{lem:subdivide}
For any graph $G$ with an edge $e$, if $G'$ is the graph obtained by subdividing $e$ then $b(G')=b(G)$.   More specifically, for any basis $\mathcal{B}$ of $\mathcal{C}(G)$ we can find a basis $\mathcal{B'}$ of $\mathcal{C}(G)$ such that all edges of $G\cap G'$ have the same charge in both bases.
\end{lem}

\begin{proof}
Write $e=uv$ and let $w$ be the subdivision vertex added in $G'$.   We have $b(G)\leq b(G')$ since $G=G'/ uw$.  To show that $b(G')\leq b(G)$, recall that we can view $G'$ as the result of replacing $e$ by a path 
$H:=u\text{-}v\text{-}w$.
We have $b(H)=0$ (since all its 2-connected components are edges) and so its $\ell$-augmented basis-number is $\ell$ by using the same path $\ell$ times.    Applying \Cref{blowup} with $\ell=b(G)$ gives $b(G')\leq b(G)$.   Furthermore, following the steps of the proof shows that the basis for $\mathcal{C}(G')$ is obtained simply by taking one of $\mathcal{C}(G)$ and replacing every occurrence of $e$ by $H$, which gives the second result.
\end{proof}

\medskip

It would be interesting to study other graph operations, such 
as the removal of an edge $e$.    It is clear that removing an edge can sometimes \emph{increase} the basis number, since $b(K_n)=3$ \cite[Theorem 1]{MR615307} but some subgraphs of $K_n$ have larger basis number.   But we do not know how much removing just one edge can increase the basis number.
\begin{op}
Is there a constant $c$ such that $b(G\setminus e)\leq b(G)+c$ for all graphs $G$ and all edges $e$?   In particular, does this hold for $c=1$?
\end{op}

We finally turn to one result that is not an obvious graph operation, but will be needed as a tool later on: we cover the graph by two subgraphs whose intersection contains a spanning tree.

\begin{lem}\label{G1G2}
Let $G$ be a connected graph with two subgraphs $G_1$ and $G_2$ such that $G_1\cap G_2$ is spanning and connected and $G=G_1\cup G_2$.
Then $b(G)\leq b(G_1)+b(G_2)$.
\end{lem}
\begin{proof}
For $i=1,2$, let $\mathcal{B}_i$ be a $b(G_i)$-basis, and define $\mathcal{B}:=\mathcal{B}_1\cup \mathcal{B}_2$.   It suffices to show that this is a generating set of $\mathcal{C}(G)$, for we can then find a basis of $\mathcal{C}(G)$ within $\mathcal{B}$, and the charge of edges in this basis is at most $b(G_1)+b(G_2)$.

Let~$T$ be a spanning tree of~$G_1\cap G_2$ and let~$\mathcal{B_T}$ denote the set of all the fundamental cycles of~$G$ with respect to~$T$. 
Since $\mathcal{B_T}$ is a basis of $\mathcal{C}(G)$, it suffices to show that all its elements are generated by $\mathcal{B}$.
So consider one fundamental cycle $C_e$ generated by
$e\notin E(T)$.   Since $G=G_1\cup G_2$, we have $e\in G_i$ for some $i\in \{1,2\}$, and therefore $C_e$ is a cycle in $G_i$.    It follows that $C_e$ is generated by $\mathcal{B}_i$, hence also by $\mathcal{B}$.
\end{proof}

\section{The basis number of  1-planar graphs}
\label{sec:1planar}

In this section, our objective is to study the basis number of 1-planar graphs.   We first show that this can be unbounded. 
In fact, it can be unbounded even for a subclass of 1-planar graphs called \defin{independent crossing 1-planar graphs} (or \defin{IC-planar graphs} for short).    These are the graphs that have a 1-planar drawing such that no two crossings have a common endpoint.

\begin{thm}\label{notbounded}
For every integer number $\ell$, there exists a $1$-planar  graph $G$ with $b(G)\geq \ell$ and maximum degree $\Delta(G)=3$.    In fact, $G$ is IC-planar.
\end{thm}

\begin{proof}
Schmeichel showed \cite[Theorem 2]{MR615307} that for any positive integer $\ell$, there exists a graph $G$ with $b(G) \geq \ell$.
The graph $G$ is not necessarily $1$-planar and does not necessarily have a maximum degree 3, but both can be achieved with simple modifications.
If $G$ has a vertex $v$ with degree $k\geq 4$, then replace $v$ by an edge $uw$ and distribute the neighbours of $v$ among $u,w$ to have a smaller degree at both.   Repeat until $\Delta=3$ and call the resulting graph $G'$.   Since $G$ can be obtained from $G'$ by contracting the inserted edges we have $b(G')\geq b(G)$ by \Cref{contract}.   Finally make $G'$ 1-planar
by subdividing edges:  Draw $G'$ arbitrarily in the plane with crossings, and then insert two new subdivision-vertices between any two consecutive crossings on an edge to get an IC-planar graph G'' that satisfies all conditions.
By \Cref{subdivide} we have $b(G'')=b(G')\geq b(G)= \ell$.
\end{proof}

In what follows, we therefore study special classes of 1-planar graphs where we \emph{are} able to give a bound on the basis number.   We immediately get one such bound from \Cref{cor:skewness}: If a 1-planar graph has at most $\ell$ crossings, then its basis number is at most $2+\ell$.
However, to get smaller constant bounds, we consider other restrictions of 1-planar graphs.



\subsection{Further definitions for 1-planar graphs}

To explain the studied subclasses of 1-planar graphs, we first need a few more definitions.
For the rest of this section, let~$G$ be a~$1$-plane graph, i.e., with a fixed 1-planar drawing.    Call an edge $e$ \defin{uncrossed} if it has no crossing; otherwise  it is \defin{crossed}.
Since we only consider 1-planar drawing, any crossed edge $e$ is crossed by exactly one other edge~$f$; call~$\{e,f\}$ a \defin{crossing pair},
and the endpoints of $e$ and $f$ are the \defin{endpoints of the crossing}.   We use the term \defin{uncrossed} also for an entire subgraph, where it means that all edges of the subgraph are uncrossed. 

The \defin{skeleton}  of a~$1$-plane graph~$G$,  denoted by~$\sk (G)$,  is the maximum uncrossed subgraph, i.e., the graph with the same set of vertices and all uncrossed edges. 

\begin{figure}[H]
\centering
{{\includegraphics[scale=0.4]{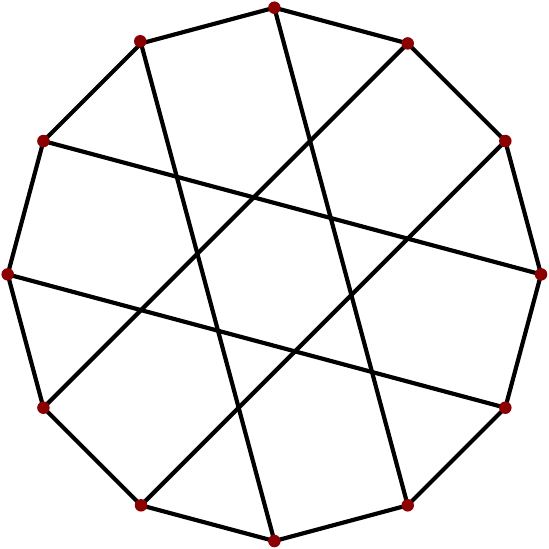}}}
\qquad
\qquad
\includegraphics[scale=0.5]{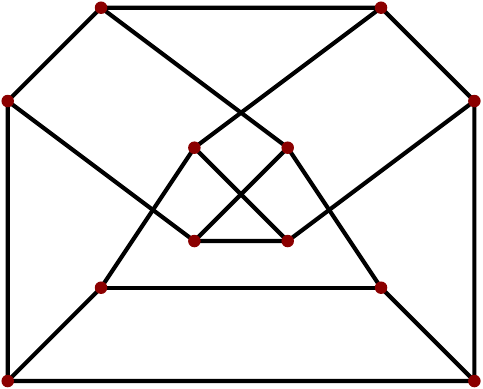}
\qquad
\qquad
\includegraphics[scale=0.5]{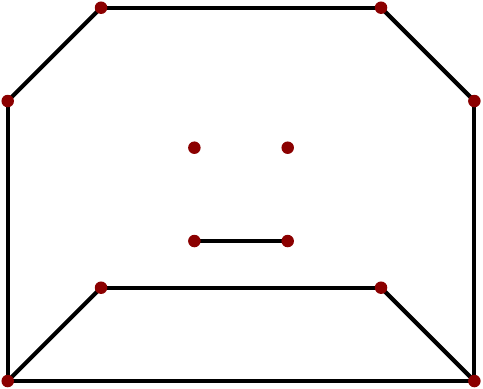}
\caption{The Franklin graph and its skeleton in a 1-planar drawing (taken from \cite{cubic1-planarity}).  
}
\label{fig:sk}
\label{fig:Franklin}
\end{figure}

The \defin{planarization} $ G^{\times} $ of $ G $ is a planar drawing created by replacing each crossing in $ G $ with a dummy vertex of degree $ 4 $. Each face of $G^{\times}$ is called a \defin{cell} of $G$; note that the boundary of a cell can hence be incident to both vertices and crossings of $G$.   An \defin{uncrossed cell} is a cell that is incident only to vertices; otherwise, it is called a \defin{crossed cell}.
Let $x$ be a crossing of $G$ and let $c_1,c_2,c_3,c_4$ be the four cells incident to $x$.
For $i=1,2,3,4$, the \defin{$i$th skirt walk} is the walk in the planarization $G^\times$ along the boundary of $c_i$ with crossing $x$ omitted.
(Thus this walk begins and ends at endpoints of different edges that cross at $x$.)   The term ``skirt walk'' is chosen because this
walk ``skirts around'' (avoids) crossing $x$.
\begin{figure}[H]
\centering
{{\includegraphics[scale=0.8]{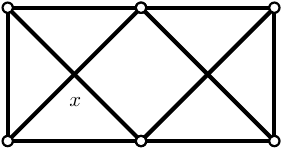}}}
\qquad
\qquad
{{\includegraphics[scale=0.8]{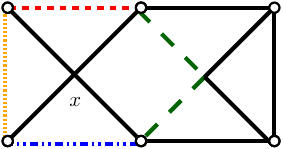}}}
\caption{The skirts of $x$ are red (dotted), orange (short dotted), blue (dot-dot-dashed), and green (dashed).}
\label{skirt_2}
\end{figure}


If we take the union of the four skirt walks, then we get a closed walk in the planarization $G^\times$, but it need not be simple, and it could use dummy-vertices corresponding to crossings.   Many of our results below are for 1-planar graphs where this union is in fact a simple cycle without dummy-vertices, which is why we introduce a name for this kind of 1-plane graph.

\begin{defn}\label{def:poppy}
In a $1$-plane graph $G$, a crossing $x$ is said to form a \defin{poppy} if the union of the four skirt-walks is a simple cycle $C$ that uses no dummy-vertices, i.e., it is a cycle in the skeleton $\sk(G)$. Cycle $C$ is then called the \defin{cycle surrounding $x$}, and the subgraph formed by $C$ and the crossing edge-pair at $x$ is called the \defin{poppy} of $x$. 
A $1$-planar graph $G$ is called a \defin{poppy $1$-planar graph} if it has a 1-planar drawing where all crossings form poppies.
\end{defn}

The concept is illustrated in \Cref{fig_p1P2P3P4}(a).    This also shows a well-known non-planar graph, the Petersen graph, which is 1-planar, and as the drawing shows, is actually poppy 1-planar.   Where convenient, we will use the term \defin{poppy 1-plane} for a 1-planar graph with a fixed 1-planar drawing where all crossings form poppies. 

\begin{figure}[H]
\centering
\subfloat[\centering]
{{\includegraphics[page=2,scale=0.8,]{figs/p1p2p3p4.pdf}}}
\qquad
\subfloat[\centering]
{{\includegraphics[scale=0.5]{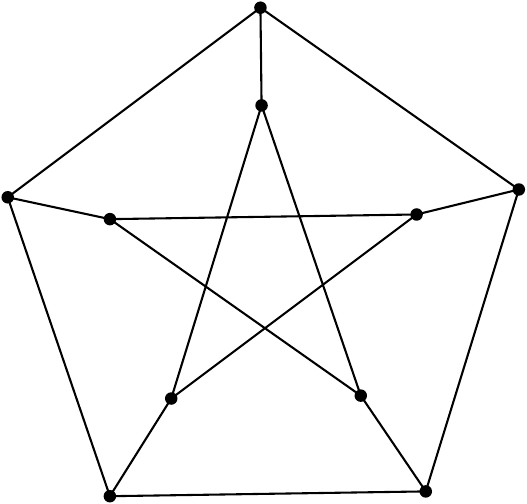}}}
\qquad
\subfloat[\centering]
{{\includegraphics[scale=0.5]{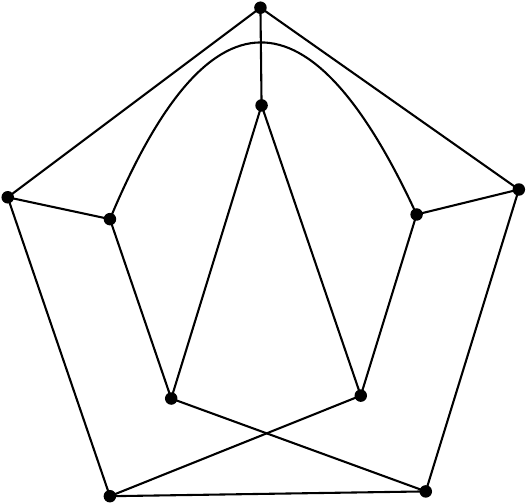}}}
\caption{(a) Crossing $x$ defines four skirt-walks which together with $x$ form a poppy.  (b) Petersen graph. (c) Petersen graph as a poppy 1-planar graph. }
\label{fig_p1P2P3P4}
\end{figure}

Poppy 1-plane graphs restrict the shape of cells near crossings.   There are other subclasses of 1-plane graphs that impose such restrictions; we review a few of them here.   Following Fabrici et al.~\cite{fabrici}, we call a 1-plane graph \defin{locally maximal} if for every crossing pair $\{uv,wy\}$, the four endpoints $u,v,w,y$ induce the complete graph $K_4$.  So in particular we then have a cycle 
$u\text{-}v\text{-}w\text{-}y\text{-}u$; 
note that this cycle \emph{could} be drawn near the crossing with uncrossed edges, but there is no requirement that it actually \emph{is} drawn this way.   In contrast to this (and inspired by the naming in \cite{karthik1}), we call a crossing $x$ \defin{full} if this cycle is actually drawn near the crossing.   
Put differently, each of the four skirt-walks of $x$ are single edges.  A \defin{full-crossing 1-plane graph} is a 1-planar graph where all crossings are full.
\Cref{skirt} illustrates the difference between full-crossing and locally-maximal 1-plane graphs.

Observe that every full-crossing 1-plane graph is also poppy 1-plane.
On the other hand, there is no relationship between locally-maximal and poppy 1-plane graph.   The Petersen graph is poppy 1-plane but not locally maximal, not even if we change the 1-planar drawing. (This holds because any drawing of the Petersen graph must have a crossing, but the Petersen graph has no $K_4$.)     The locally maximal graph in \Cref{skirt}(b) consists of three copies of $K_8-M$ (i.e., $K_8$ with a perfect matching removed),
with one edge from each identified.    
One can argue that 
\emph{any} 1-planar drawing of this graph has a cell that is incident to at least two crossings, and so it is not poppy 1-planar.

\begin{figure}[H]
\centering
\subfloat[\centering]
{{\includegraphics[scale=0.5]{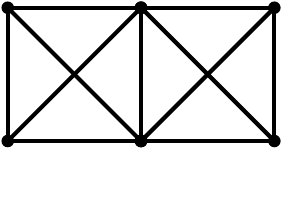}}}
\qquad
\qquad
\subfloat[\centering]
{{\includegraphics[page=2,scale=0.5]{figs/notlocalmax.pdf}}}
\caption{A full-crossing 1-plane graph and a graph that is locally maximal 1-plane, but some crossings (e.g.~the ones on the unbounded cell) are not full.}
\label{skirt}
\end{figure}


We will later need a result about the skeletons of our graph classes.

\begin{lem}\label{lem:connected_skeleton}
Let $G$ be a connected graph.   If $G$ is locally maximal 1-plane or poppy 1-plane, then (possibly after redrawing some edges without adding crossings) it has a connected skeleton.
\end{lem}
\begin{proof}
We first explain how to redraw edges, which will be needed only if $G$ is locally maximal 1-plane.  Assume that edges $u_0u_2$ crosses edge $u_1u_3$ at $x$.   Then by local maximality $G$ contains a cycle $u_0\text{-}u_1\text{-}u_2\text{-}u_3\text{-}u_0$.   Each edge $u_iu_{i+1}$ (addition modulo 4)
could be drawn without crossing by walking near $x$.    Therefore, after appropriate redrawing, we can assume that at any crossing pair $\{u_0u_2,u_1u_3\}$, and for any $i\in \{0,\dots,3\}$, there exists a path from $u_i$ to $u_{i+1}$ in the skeleton.   (Note that this condition holds automatically in any poppy 1-plane drawing, since we can use the skirt-walk.)

Now we argue that the skeleton is connected by showing that for any two vertices $s,t$ of $G$ there exists an $st$-path in $sk(G)$.
We know that there exists an $st$-path $P$ in $G$ since $G$ is connected.   If $P$ uses a crossed edge $uv$ (say it is crossed by $wy$), then we can remove $uv$ from $P$ and replace it by a path from $u$ to $w$ and a path from $w$ to $v$; by the above, there are such paths in the skeleton.
Repeating at all crossings in $P$ gives an $st$-path in the skeleton.
\end{proof}

Our final class of 1-planar graphs is the first class of 1-planar graphs ever studied (by Ringel in 1965 \cite{Ringel1965}).
It is known that every $1$-planar graph with $ n $ vertices has at most $ 4n - 8 $ edges (see \cite{MR0732806}). An \defin{optimal $1$-planar graph} is a simple 1-planar  with exactly $ 4n - 8 $ edges.
In any 1-planar drawing of such a graph, all cells must be incident to one crossing and two vertices, which in particular means that the graph is full-crossing 1-plane.

\subsection{Skeleton-properties}

Now we are ready to give upper bounds on the basis number of some 1-plane graphs with special properties.


\begin{prop}\label{connected_skeleton}
Let $G$ be a $1$-plane graph for which $\sk(G)$ is connected. Then $b(G)\leq 4$.
\end{prop}
\begin{proof}
Let~$\{(e_i,f_i)\mid 1\leq i\leq\ell \}$ be the crossing pairs of~$G$, and note that each edge of $G$ can only belong to at most one of these pairs.
Consider the two graphs~$G_1=\sk(G)\cup\{e_i \mid 1\leq i\leq\ell\}$ and~$G_2=\sk(G)\cup \{f_i \mid 1\leq i\leq\ell\}$;
these are planar since they contain only one of each pair of crossing edges.
By \cref{maclane}, there is a~$2$-basis~$\mathcal B_i$ of~$G_i$ for~$i\in \{1,2\}$.   We designed \cref{G1G2} exactly for this situation: $G=G_1\cup G_2$, and the intersection $G_1\cap G_2$ (which is the skeleton $sk(G)$) is spanning by definition of a skeleton and connected by assumption.
By \cref{G1G2} 
therefore $b(G)\leq 4$.
\end{proof}

\begin{thm}
\label{thm:localmaxB4}
If $G$ is a locally maximal 1-plane graph or a poppy 1-plane graph, then $b(G)\leq 4$. 
\end{thm}
\begin{proof}
We may assume that $G$ is connected, by applying the result to each connected component of $G$.  
Since $G$ is connected, the skeleton is also connected (\Cref{lem:connected_skeleton}), and the result follows from \Cref{connected_skeleton}.
\end{proof}

\paragraph{Disconnected skeletons.}

We now expand the class of 1-planar graphs for which the
basis number is bounded, by combining subgraphs where the skeleton is connected.


Consider a 1-planar graph $G$ that is connected,  but its skeleton $\sk(G)$ is not connected (otherwise $G$ has a 4-basis by \cref{connected_skeleton}).    
Our goal is to impose conditions under which we can apply
\cref{G1G2}, i.e., cover $G$ with two subgraphs $G_1,G_2$ that are spanning, have a connected skeleton, and $G_1\cap G_2$ is connected.
To this end, we want three disjoint edge set $E_1,E_2,E_3 \subseteq E$ that each contain a spanning tree.
We will include $E_3$ in both $G_1$ and $G_2$ 
(this ensures connectivity of $G_1\cap G_2$).   Edge sets $E_1$ and $E_2$ will be chosen carefully so that excluding crossed edges of $E_i$ from $G_i$ gives (for $i=1,2$) a connected skeleton.   To describe the exact conditions that must be met, we define an auxiliary graph first.

\begin{defn}
\label{def:Q}
Let~$G$ be a~$1$-plane graph. The \defin{auxiliary graph}~$Q$ is defined as 
follows:\footnote{Graph $Q$ is roughly the quotient graph with respect to the connected components of $\sk(G)$, hence the choice of name $Q$.}
For each connected component $X_i$ of $\sk(G)$, there is a vertex $x_i$ in $V(Q)$. 
For any pair $e,f$ of crossing edges in $G$, if both edges connect a vertex of $X_i$ with a vertex of $X_j$ (for some components $X_i,X_j$), then we add a double-edge $(x_i,x_j)$ in $Q$. We call these the edges \emph{corresponding} to $e$ and $f$.
\end{defn}

\begin{figure}[H]
\centering
\subfloat[Traditional]
{{\includegraphics[scale=0.37,page=2]{figs/Desargues_graph.pdf}}}
\qquad
\subfloat[\centering A 1-plane drawing]
{{\includegraphics[scale=0.4,page=4]{figs/Des_1_planar.pdf}}}
\qquad
\subfloat[\centering Its auxiliary graph $Q$]
{{\includegraphics[scale=0.4]{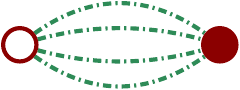}}}
\caption{
The Desargues graph, with a 1-planar drawing (inspired by the one from \cite{cubic1-planarity}) with the connected components of the skeleton indicated, and its 
auxiliary graph that has four edge-disjoint spanning trees.
}
\label{fig:aux_graph}
\label{fig:Desargues}
\end{figure}

\begin{prop}\label{firstmain} 
\label{lem:3tree}
Let~$G$ be a 1-plane graph for which the auxiliary graph $Q$ contains three 
disjoint spanning trees.
Then~$G$ has an~$8$-basis.
\end{prop}

\begin{proof}
Let $T_i = (V(Q), \hat{E}_i)$ for $i=1,2,3$ be three spanning trees of $Q$.
For $i=1,2$, let $E_i$ be the edges of $G$ that correspond to $\hat{E}_i$, and define $G_i$ to be the graph 
$G\setminus E_i$.    
We claim that $G_1$ and $G_2$ satisfy all conditions of \cref{G1G2}:
\begin{itemize}
\item $G_1\cap G_2$ contains all edges of $\sk(G)$ and all edges $E_3$ that correspond to edges $\hat{E}_3$.   Since
$\hat{E}_3$ forms a spanning tree between the connected components of $\sk(G)$, therefore $G_1\cap G_2$ is connected.
\item We claim that $G_i$ (for $i=1,2$) has a connected skeleton.   
To see this, consider any cut $(A,B)$ (i.e., partition the vertices into two disjoint sets $A$ and $B$); it suffices to show that some edge of $G_i$ connects $A$ to $B$ and is uncrossed in $G_i$.

Clearly this holds if some connected component of $\sk(G)$ intersects both $A$ and $B$, so assume not.   Therefore  $(A,B)$ corresponds to a cut $(A_Q,B_Q)$ of the vertices of $Q$.    Since $T_i$ is a spanning tree, some edge $\hat{e}$ 
in $T_i$ connects $A_Q$ and $B_Q$.
This edge exists in $Q$ because there was a pair $\{e,f\}$ of crossing edges in $G$ that gave rise to $\hat{e}$ and a parallel edge $\hat{f}$.   Edge $e$ connected some component $X_A$ of $\sk(G)$ that lies in $A$ with a component $X_B$ of $\sk(G)$ that lies in $B$. Since $f$ connects the same components (by definition of $Q$) edge $f$ also connects $A$ to $B$.

Since $T_i$ is a tree while $\hat{e}\cup \hat{f}$ is a cycle in $Q$, edge $\hat{f}$ is \defin{not} in $T_i$. So $f\not\in E_i$ belongs to $G_i$.   Since $f$ was crossed by $e$, it was not crossed by any other edge by 1-planarity, and so it is uncrossed in $G_i$ (which excludes $e$).   So the uncrossed edge $f$ belongs to cut $(A,B)$ as desired.
\end{itemize}

It follows from \Cref{connected_skeleton} that each $G_i$ has a $4$-basis.
Now, applying \Cref{G1G2}, we deduce that $G$ has an $8$-basis.
\end{proof}

\subsection{Crossing-properties}

We now study the basis number for poppy 1-plane graphs (see \Cref{def:poppy}). 
Recall that these include full-crossing 1-plane graphs as special case: a full-crossing 1-plane graph is a poppy 1-plane graph where every poppy of a crossing is $K_4$. So we will first prove a result for full-crossing 1-plane graphs (where having only $K_4$'s is helpful) and then generalize.

\subsubsection{1-planar graphs with full crossings}

We build up a basis for full-crossing 1-plane graphs in three steps:   first, we argue that we can obtain a generating set by starting with the faces of the skeleton and adding a basis for the $K_4$-poppies, then we show that $K_4$ has bases with special properties, and then we show that therefore merging such bases into one of the skeleton in a suitable way gives a 3-basis.

\begin{lem}
\label{lem:addCrossingBases}
Let $G$ be a full-crossing 1-planar graph.    Let $\mathcal{B}_{sk}$ be a generating set of the skeleton. For each crossing $x$, let $\mathcal{B}_x$ be a basis of the $K_4$-poppy at $x$.  Define $\mathcal{B}$ to be the set obtained by removing from $\mathcal{B}_{sk}$ any cycle that surrounds a crossing, 
and then adding $\mathcal{B}_x$ for all crossings $x$. Then $\mathcal{B}$ generates $\mathcal{C}(G)$.
\end{lem}


\begin{figure}[H]
\centering
\includegraphics[width=1\linewidth]{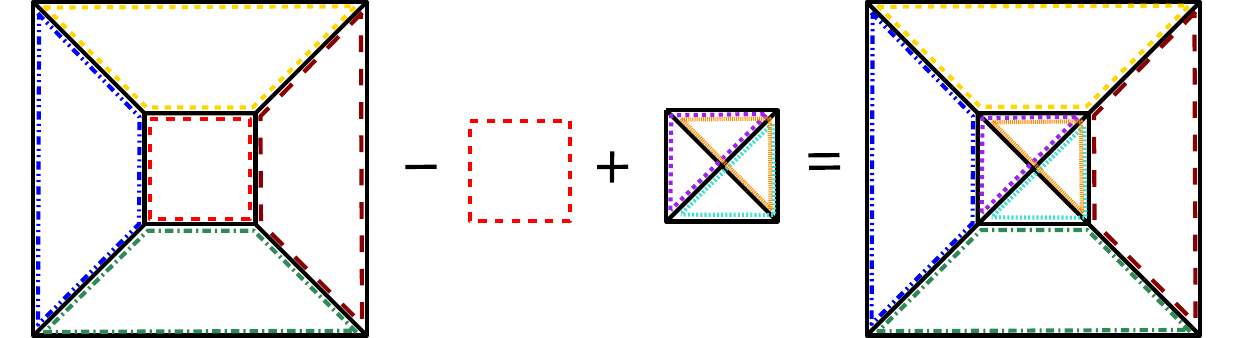}
\caption{Removing the cycle of skeleton-edges that surround the crossing and add three cycles of $K_4$.}
\label{fig:sk_basis}
\end{figure}
\begin{proof}
Let $C$ be any cycle in $G$, and let $e_1,\dots,e_k$ be the edges of $C$ that are crossed, say with crossings $x_1,\dots,x_k$.    For $i=1,\dots,k$, let $C_i$ be the cycle of skeleton-edges that surrounds crossing $x_i$, and let $P_i$ be a path within $C_i$ that connects the ends of $e_i$. We have
$$ C = \big(C + \sum_{i=1}^k (- e_i + P_i)\big)  + \sum_{i=1}^k (P_i+e_i) = C' + \sum_{i=1}^k (P_i+e_i)$$
where $C'$ is obtained by starting with $C$, remove all crossed edges $e_1,\dots,e_k$, and replacing each $e_i$ with a path $P_i$ in the skeleton between its endpoints.
So $C'$ can be generated by $\mathcal{B}_{sk}$.     If this generation uses a cycle $C_i$ that surrounds a crossing, then $C_i$ is not necessarily in $\mathcal{B}_{sk}$, but 
it can be generated by $\mathcal{B}_{x_i}$ and so $C'$ can be generated with $\mathcal{B}$.    Also $P_i+e_i$ is a cycle in the $K_4$-poppy at $x_i$ and so can be generated with $\mathcal{B}_{x_i}$, so we can generate $C$ with $\mathcal{B}$.
\end{proof}

\begin{lem}
\label{genK4}
\label{lem:goodBase}
Consider a graph $K_4$ drawn with one crossing $x$, and let $C$ be the cycle surrounding $x$.   For any assignment of $\{1,1,2,2\}$ to the edges of $C$, there is a basis of $K_4$ such that the charge of each edge in $C$ is at most the assigned number, and all other edges have charge at most 3.
\label{clm:K4}
\end{lem}
\begin{proof}
Enumerate cycle $C$ as 
$u_0\text{-}u_1\text{-}u_2\text{-}u_3\text{-}u_0$ 
such that $e:=u_0u_1$ was assigned 1.    If the other edge $e'$ that was assigned 1 is $u_1u_2$, then let $T$ be the spanning tree defined by the three edges incident to $u_3$.   Symmetrically, if $e'$ is $u_3u_0$, then let $T$ be the spanning tree defined by the edges incident to $u_2$.    Finally, if $e'=u_2u_3$ then let $T$ be the spanning tree defined by the path 
$u_1\text{-}u_2\text{-}u_0\text{-}u_3$. 
One easily verifies (see also \Cref{K4}) that the fundamental basis defined by $T$ uses $e$ and $e'$ only once and the other edges of $C$ twice.
\end{proof}

\begin{figure}[ht]
\hspace*{\fill}
\includegraphics[page=1]{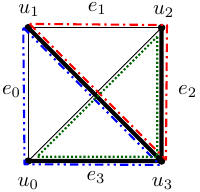}
\hspace*{\fill}
\includegraphics[page=3]{figs/K4.pdf}
\hspace*{\fill}
\includegraphics[page=2]{figs/K4.pdf}
\hspace*{\fill}
\caption{Possible bases for a $K_4$ obtained by choosing suitable spanning trees (bold).} 
\label{K4}
\end{figure}
\begin{lem}
\label{lem:goodOrientation}
Let $G$ be a full-crossing 1-plane graph.    Then $sk(G)$ has an orientation such that for every 4-cycle $C$ that surrounds a crossing, exactly two edges are directed clockwise and two edges are directed counter-clockwise.
\end{lem}
\begin{proof}
Define the \defin{dual graph} $D$ of the skeleton to have one vertex for every face and add an edge between two face-vertices $v_F$ and $v_{F'}$ whenever $F$ and $F'$ share an edge.    
It is well-known that any graph has an orientation such that any vertex $v$ has at most $\lceil \deg(v)/2 \rceil$ incoming and outgoing edges (see e.g.~\cite{Nash-Williams_1960}).     
Fix such an orientation of $D$, and turn it into
an orientation of $sk(G)$ by directing each edge of $sk(G)$ so that it crosses the corresponding edge of $D$ left-to-right.    

Any cycle $C$ that surrounds a crossing is (in a full-crossing 1-plane graph) a 4-cycle that bounds a face of the skeleton; let $v_C$ be the corresponding face-vertex in $D$.  In the orientation, $v_C$ has two incoming and two outgoing edges, and so $C$ has two clockwise and two counterclockwise edges. 
\end{proof}

Now we combine all these results to bound the basis number of full-crossing
1-planar graphs.

\begin{thm}
\label{localmax}
Let $G$ be a 2-connected full-crossing~$1$-planar graph.
Then $b(G)\leq 3$.
\label{thm:localmax}
\end{thm}
\begin{proof}
Let $x_1,\dots,x_k$ be the crossings of $G$, and for $i=1,\dots,k$ let $C_i$
be the cycle that surrounds $x_i$.
Using \Cref{lem:goodOrientation}, we can find an orientation of
$sk(G)$ such that each $C_i$ has exactly two clockwise and two 
counter-clockwise edges.        For the following discussion, it will help
to view the drawing of $G$ as a drawing on the sphere; this exchanges the
meaning of ``clockwise'' and ``counter-clockwise'' for the face of the skeleton
that used to
be the outer-face, but still, each $C_i$ has exactly two clockwise edges.

For $i=1,\dots,k$, choose a basis $\mathcal{B}_{x_i}$ for the $K_4$-poppy at $x_i$ by assigning 1 to the clockwise and 2 to the counter-clockwise edges of $C_i$. Then use the basis of \Cref{lem:goodBase} whose charges along $C$ respect this assignment. 
Now define a generating set $\mathcal{B}$ of $\mathcal{C}(G)$ as in \Cref{lem:addCrossingBases}, i.e., start with the set $\mathcal{B}_{sk}$ of face boundaries of $sk(G)$, remove $C_1,\dots,C_k$, and add $\mathcal{B}_1,\dots,\mathcal{B}_k$.    We claim that any edge $e$ has charge at most 3.     This holds if $e$ is crossed (say at $x_i$), because then it is used only by $\mathcal{B}_i$.
If $e$ is uncrossed, then let $F,F'$ be the two incident faces in $sk(G)$, named such that $e$ is clockwise for $F$ and counter-clockwise for $F'$.  Face $F$ contributes to $\mathcal{B}$ either via $\mathcal{B}_{sk}$ (if $F$ contains no crossing and so the boundary of $F$ is in $\mathcal{B}_{sk}$), or via $\mathcal{B}_i$ (if $F$ contains crossing $x_i$).    Either way, it adds a charge at most 1 to $e$ since $e$ is clockwise for $F$.     Similarly, $F'$ adds charge at most 2 to $e$.    No other elements of $\mathcal{B}$ uses $e$, so $ch_{\mathcal{B}}(e)\leq 3$.
\end{proof}

Since optimal 1-planar graphs are full-crossing maximal, we get a result for them as well.

\begin{cor}\label{optimal}
Every optimal~$1$-planar graph $G$ satisfies $b(G)\leq 3$.
\end{cor}

Note that we actually have some choices in the basis for \Cref{thm:localmax}.   In particular $\mathcal{B}_{sk}$ can omit one of the face boundaries and is still a basis.   Also, there are many possible orientations of $sk(G)$, and for each of them the reverse orientation can also be used.   Exploiting this, one can show for example that for any edge $e$ that is either crossed or incident to one uncrossed cell, we can find a 3-basis of the graph where $ch(e)\leq 1$. Details are left to the reader.

\subsubsection{Poppy 1-planar graphs}

Now we turn to poppy 1-planar graphs.    Almost all the previous lemmas carry over, with nearly the same proof, with the exception of \Cref{lem:goodOrientation} (which details how to orient the skeleton).   We can therefore prove that poppy 1-planar graphs have a 3-basis only under the additional assumption that such an orientation exists.
The proof of the following lemma is verbatim the proof of \Cref{lem:addCrossingBases}, except that ``$K_4$-poppy'' now is simply ``poppy''.

\begin{lem}
\label{lem:addCrossingBasesPoppy}
Let $G$ be a 2-connected poppy 1-plane graph.    Let $\mathcal{B}_{sk}$ be a generating set of the skeleton. For each crossing $x$, let $\mathcal{B}_x$ be a basis of the poppy at $x$.  Define $\mathcal{B}$ to be the set obtained by removing from $\mathcal{B}_{sk}$ any cycle that surrounds a crossing, 
and adding $\mathcal{B}_x$ for all crossings $x$. Then $\mathcal{B}$ generates $\mathcal{C}(G)$.
\end{lem}

We can also again have special bases for a poppy.

\begin{lem}
\label{lem:goodBasePoppy}
Let $G$ be a poppy 1-plane graph, and let $x$ be a crossing.
For any assignment of $\{1,1,2,2\}$ to the skirt walks of $x$, we can find a basis of the poppy of $x$
such that each edge in each skirt walk has charge at most the assigned number, and all other edges have charge at most 3.
\end{lem}
\begin{proof}
The poppy $P$ can be viewed as graph $K_4$, drawn with one crossing, where the uncrossed edges have been subdivided to become skirt-walks.    The result now follows from \Cref{lem:goodBase} together with the result from \Cref{subdivide} that we can subdivide edges without changing the charges in a given base.
\end{proof}

As mentioned, \Cref{lem:goodOrientation} does not generalize to poppy 1-planar graphs, and so we give a name to the situation where the conclusion holds.

\begin{defn}
Let $G$ be a poppy 1-plane graph.  An \defin{balanced orientation of the skirt-walks} is an orientation of the edges of $sk(G)$ such that every skirt-walk is oriented from one end to the other, and for any crossing $x$, exactly two skirt-walks of $x$ are oriented clockwise and exactly two skirt-walks of $x$ are oriented counter-clockwise. 
\end{defn}

To see some examples, consider \Cref{fig:goodOrientation}: In these poppy 1-planar drawings, $K_{3,4}$ has a balanced orientation while the 
Petersen graph does not. 
The latter can be seen with a simple propagation-argument: 
if two skirt-walks $W,W'$ share an edge then they must be oriented in opposite directions.    In this embedding of the Petersen graph, we can find five skirt-walks where each shares an edge with the next one.   This forces three skirt-walks of one crossing to be all oriented the same, and so no balanced orientation can exist.

\begin{figure}[H]
\centering
\subfloat[$K_{3,4}$]{
    \includegraphics[scale=0.8]{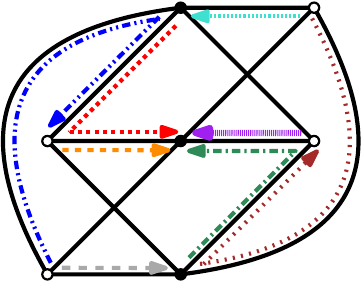}}
\qquad
\qquad
\subfloat[Peterson graph]
{{\includegraphics[scale=0.6]{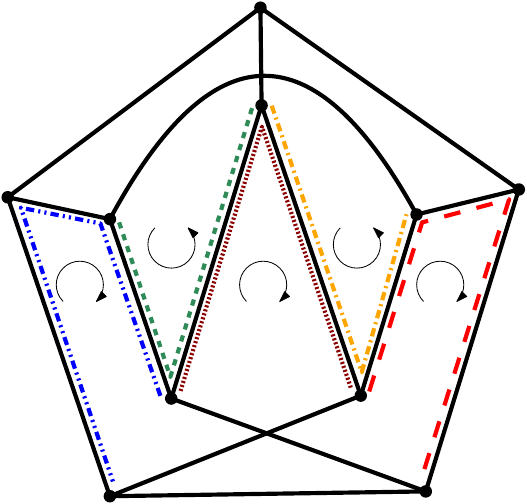}}}
\caption{This poppy 1-planar drawing of  $K_{3,4}$ has a balanced orientation, while the one of the Petersen graph does not.
        }
\label{fig:goodOrientation}
\end{figure}
The proof of the following result is almost exactly the same as for \Cref{thm:localmax}, but we repeat parts of it since ``edge'' must be replaced by ``skirt-walk'' sometimes.

\begin{thm}
\label{thm:poppy}
Let $G$ be a poppy 1-plane graph that has a balanced orientation of the skirt-walks.
Then $b(G)\leq 3$.
\end{thm}
\begin{proof}
Let $x_1,\dots,x_k$ be the crossings of $G$, and for $i=1,\dots,k$ let $C_i$
be the cycle that surrounds $x_i$.
We again view the drawing of $G$ as a drawing in the sphere, and have by assumption an orientation of the skeleton such that at each crossing exactly two of the four skirt-walks are oriented clockwise.

For $i=1,\dots,k$, choose a basis $\mathcal{B}_{x_i}$ for the poppy at $x_i$ by assigning 1 to the clockwise and 2 to the couter-clockwise skirt walks at $x$.
Then use the basis of \Cref{lem:goodBasePoppy} whose charges along the skirt walks respect this assignment. 
Now define a generating set $\mathcal{B}$ of $\mathcal{C}(G)$ as in \Cref{lem:addCrossingBasesPoppy}, i.e., start with the set $\mathcal{B}_{sk}$ of face boundaries of $sk(G)$, remove $C_1,\dots,C_k$, and add $\mathcal{B}_1,\dots,\mathcal{B}_k$.    Exactly as in \Cref{thm:localmax} one argues that every edge has a charge at most 3.
\end{proof}

\subsubsection{Independent crossings} 

\Cref{thm:poppy} is a bit disappointing in that we require a given balanced orientation of the skirt-walks.   Not every poppy 1-plane graph has such an orientation, and we do not know how easy it is to test whether it has one.    But in this section, we give some conditions under which such an orientation exists.

Recall that we already defined the \emph{IC-planar} graphs as the 1-plane graphs where no two crossings share an endpoint.    A graph class between IC-planar and 1-planar graphs are the \defin{near-independent-crossing
                                                                                                                                                                                                 1-planar graphs} (or \defin{NIC-planar graphs} for short), where no two crossings share two endpoints \cite{BACHMAIER201723}.
For both IC-planar and NIC-planar the basis number is unbounded.   We now define a subclass of poppy 1-planar graphs that also use the idea of near-independent crossings, but now we use the skirt-walks for measuring independence, rather than the endpoints of the crossing.

\begin{defn}
Let $G$ be a poppy 1-plane graph.   We say that $G$ has \defin{near-independent skirt walks} if no two skirt walks of two different crossings have two vertices in common.
\footnote{For our result it would actually suffice to demand that no two skirt-walks of different crossings have an edge in common, but to keep the similarity to NIC-planar graphs we only demand that they do not have two common vertices.}
\end{defn}

\begin{thm}\label{inc_cross_cell}
Let $G$ be a poppy $1$-plane graph with near-independent skirt walks.
Then ${b(G)\leq 3}$.
\end{thm}
\begin{proof} Since the drawing has near-independent skirt walks, no edge belongs to the skirt-walks of two different crossings, so we can orient the skirt-walks of each crossing independently of the others.   Since the graph is poppy 1-plane, at each crossing the union of the skirt walks is a simple cycle, and it is therefore straightforward to find a balanced orientation of the skirt walks.
\end{proof}

\subsection{More examples and questions}
\label{subse:story}

We have now seen a great number of subclasses of 1-planar graphs that have a 3-basis.   Of course, we know from \Cref{notbounded} that not all 1-planar graphs have a 3-basis.   So this raises a natural open question:

\begin{op}
What is the smallest 1-planar graph that does not have basis number~3?   
\end{op}

We considered a large number of ``obvious'' candidates for this open problem.
\Cref{ta:story} gives an overview of graphs that we studied, as well as properties of their 1-planar drawings.  

\begin{table}[H]
\centering
\begin{tabular}{|l|c|c|c|l|}
\hline
\textbf{Graph} & \textbf{Poppy} & \textbf{Loc Max} & \textbf{Conn Skel} & \textbf{Basis Number} \\ \hline\hline
$K_{6}$ (\Cref{fig:K6}) & \cmark   &  \cmark  &  \cmark &3 (\cite[Thm.~1]{MR615307})\\
\hline
$K_{3,4}$ (\Cref{fig:goodOrientation}(a)) & \cmark   &  \xmark  &  \cmark &3 (\cite[Thm.~3]{MR615307})\\
\hline
$K_{4,4}$ (\Cref{fig:K44}) & \xmark   &  \xmark  &  \xmark &3 (\cite[Thm.~3]{MR615307})\\
\hline
4d-hypercube (\Cref{fig:q_4}) & \xmark   &  \xmark  &  \cmark &3 (\cite[Thm.~5]{MR615307})\\
\hline
Petersen (\Cref{fig:goodOrientation}(b)) & \cmark  & \xmark  & \xmark & 3 (\cite[Thm.~3.1]{MR1983447})\\
\hline
Heawood (\Cref{fig:no3basis1}) & \cmark   &  \xmark  &  \cmark &3 (\cite[Thm.~3.2]{MR1983447})\\
\hline
McGee graph (\Cref{fig:McGee}) & \xmark   &  \xmark  &  \xmark & 3 (\cite[Thm.~3.3]{MR1983447}) \\
\hline
Nauru (\Cref{fig:Nauru}) & \xmark   &  \xmark  &  \xmark & 3 \cite{lehner2024basis}\\
\hline
Franklin graph (\Cref{fig:Franklin}) & \xmark   &  \xmark  &  \xmark & 3 \cite{lehner2024basis} \\
\hline
Desargues (\Cref{fig:Desargues}) & \xmark   &  \xmark  &  \xmark &3  (\Cref{prop:des_basis})\\
\hline
Subdivided 8-cage (\Cref{fig:no3basis2}) & \xmark   &  \xmark  &  \xmark &4 (\Cref{prop:sub_tutte_b_4})\\
\hline
\end{tabular}
\caption{Various 1-planar graphs, their properties (poppy, locally maximal, connected skeleton) in the 1-planar drawings that we used, and their basis numbers.
        }
\label{tab:graph_properties}
\label{ta:story}
\end{table}

We selected graphs as follows.   First, we considered small graphs
that are not planar but whose interesting properties (such as being symmetric or a so-called cage or Levi graph)
means that their basis number has been studied before (see e.g.~\cite{MR615307,MR1983447}).
Alas, for all the ones where we could find 1-planar drawings, the basis number was 3 and
so they did not provide the desired example.   (We also note here that many of these
graphs are in fact toroidal, and so the 3-basis could be found more easily with \cite{lehner2024basis};
we illustrate such a 3-basis in the figures below.)
\begin{figure}[H]
\centering
\subfloat[$K_6$]{
    \includegraphics[page=1,scale=0.6]{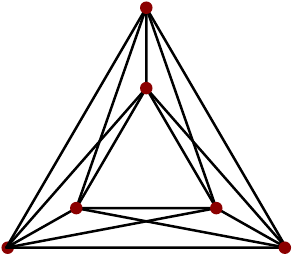}
}
\qquad
\qquad
\subfloat[$K_{4,4}$]{
    \includegraphics[page=2,scale=0.6]{figs/small_complete.pdf}
}
\caption{1-planar drawings of $K_6$ and $K_{4,4}$.}
\label{fig:K6}
\label{fig:K44}
\end{figure}

\begin{figure}[H]
\centering
\subfloat[Traditional]{
    \includegraphics[page=2,scale=0.4]{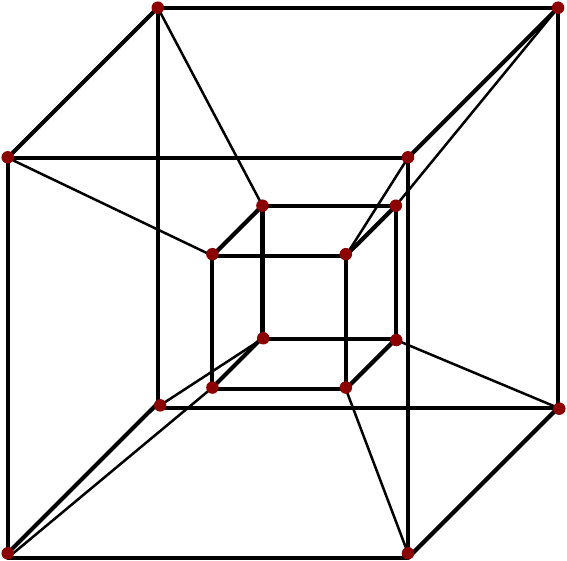}
}
\qquad
\subfloat[$1$-planar]{
    \includegraphics[page=5,scale=0.4]{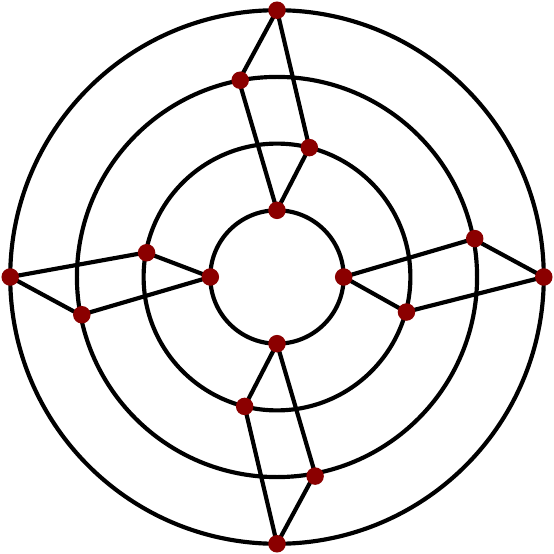}
}
\qquad
\subfloat[Toroidal embedding]{
    \includegraphics[page=3,scale=0.6]{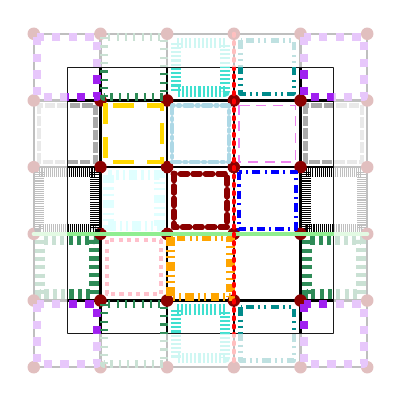}
}
\caption{Drawings of the 4d-hypercube.}
\label{fig:q_4}
\label{fig:4dHypercube}
\end{figure}

\begin{figure}[H]
\centering 
\subfloat[Traditional]
{{\includegraphics[scale=0.4,page=3]{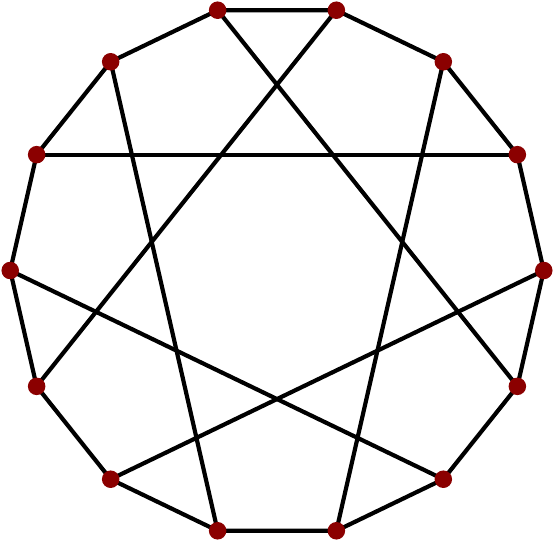}}}
\qquad
\subfloat[$1$-planar]
{{\includegraphics[scale=0.4,page=4]{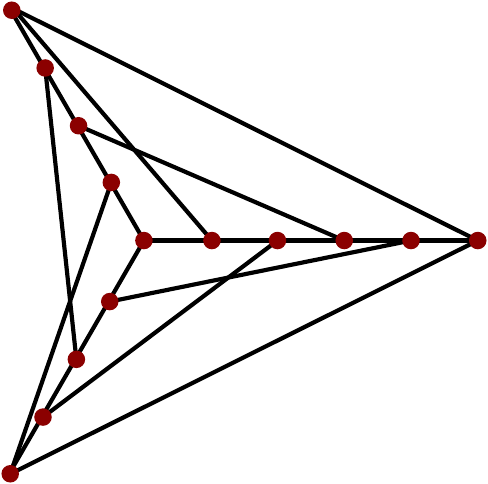}}}
\subfloat[Toroidal embedding] 
{{\includegraphics[scale=0.6,page=2,trim=0 50 0 50,clip]{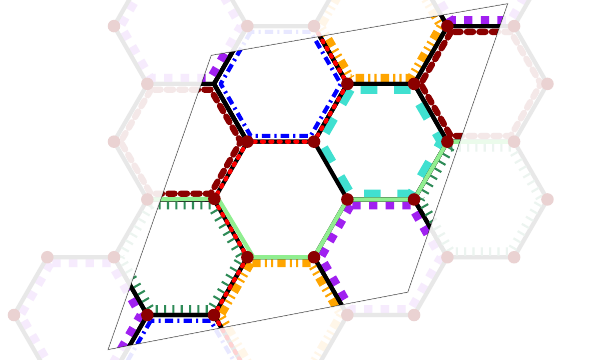}}}
\caption{ Drawings of the Heawood graph.
        }
\label{fig:no3basis1}
\label{fig:Heawood}
\end{figure}

\begin{figure}[H]
\centering 
\subfloat[Traditional]
{{\includegraphics[scale=0.5]{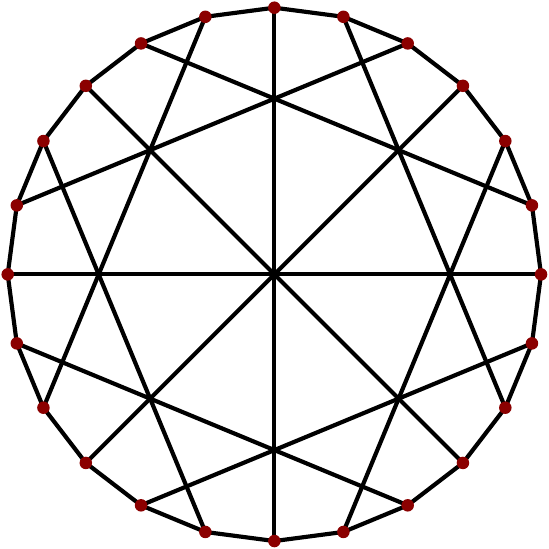}}}
\qquad
\subfloat[1-planar \cite{cubic1-planarity}]
{{\includegraphics[scale=0.65]{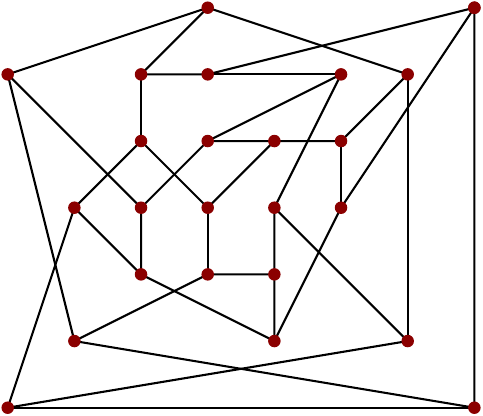}}}
\caption{Drawings of the McGee graph.}
\label{fig:McGee}
\end{figure}

Next, we looked at some graphs whose 1-planarity has been demonstrated for other reasons
(see \cite{cubic1-planarity}, we repeat the 1-planar drawings here for completeness).
Two of these (the Nauru graph and the Franklin graph) are actually honeycomb toroidal graphs (see \cite{Alspach21}); in particular, therefore they are toroidal and therefore have basis number 3 \cite{lehner2024basis}.

\begin{figure}[H]
\centering 
\subfloat[Traditional]
{{\includegraphics[scale=0.4]{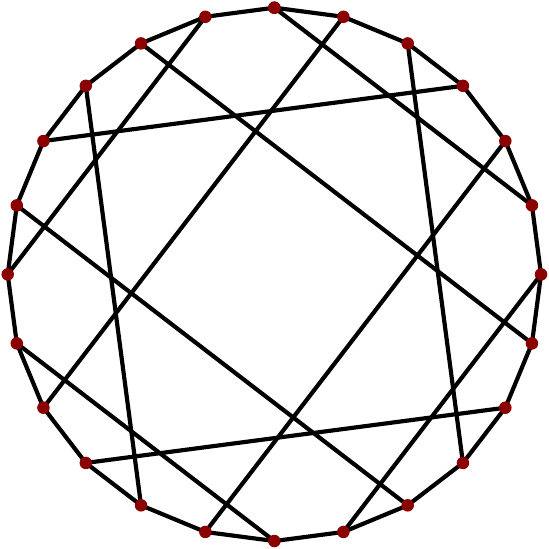}}}
\qquad
\subfloat[1-planar \cite{cubic1-planarity}]
{{\includegraphics[scale=0.5]{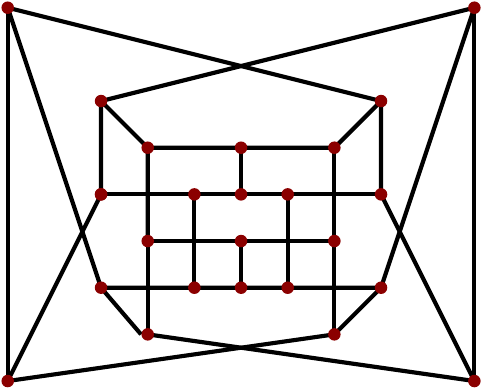}}}
\qquad
\subfloat[Toroidal embedding]
{{\includegraphics[page=3,scale=0.5,trim=0 100 0 100,clip]{figs/Nauru.pdf}}}
\caption{Drawings of the Nauru graph.  
        }
\label{fig:Nauru}
\end{figure}

\tbcomment{We have the Franklin graph earlier now, and so we do not need it here again.}



But the Desargues graph
(\Cref{fig:Desargues,fig:Desargues_basis})
is known to be not toroidal.
From \Cref{lem:3tree} we know that its basis number is at most 8, but with a direct construction we can actually show that it has a 3-basis.

\begin{prop}\label{prop:des_basis}
The basis number of the Desargues graph is $3$.
\end{prop}
\begin{proof}
To define a basis for this graph, fix the drawing of \Cref{fig:Desargues_basis}(a), and classify an edge as \emph{outer}/\emph{spike}/\emph{inner} if it has exactly two/one/zero endpoints on the infinite cell.  Define basis $\mathcal{B}$ to consist of the 10-cycle $D_{\text {inner}}$ formed by the inner edges, as well as the ten cycles $D_e$ obtained by taking an inner edge $e$, continuing along the two spikes at its endpoints, and closing up into a 6-cycle by connecting the ends of the spikes via outer edges.
One readily verifies that every edge has a charge at most 3, so we must only show that $\mathcal{B}$ is a generating set.

\begin{figure}[H]
\centering 
\subfloat[Basis]
{{\includegraphics[scale=0.4]{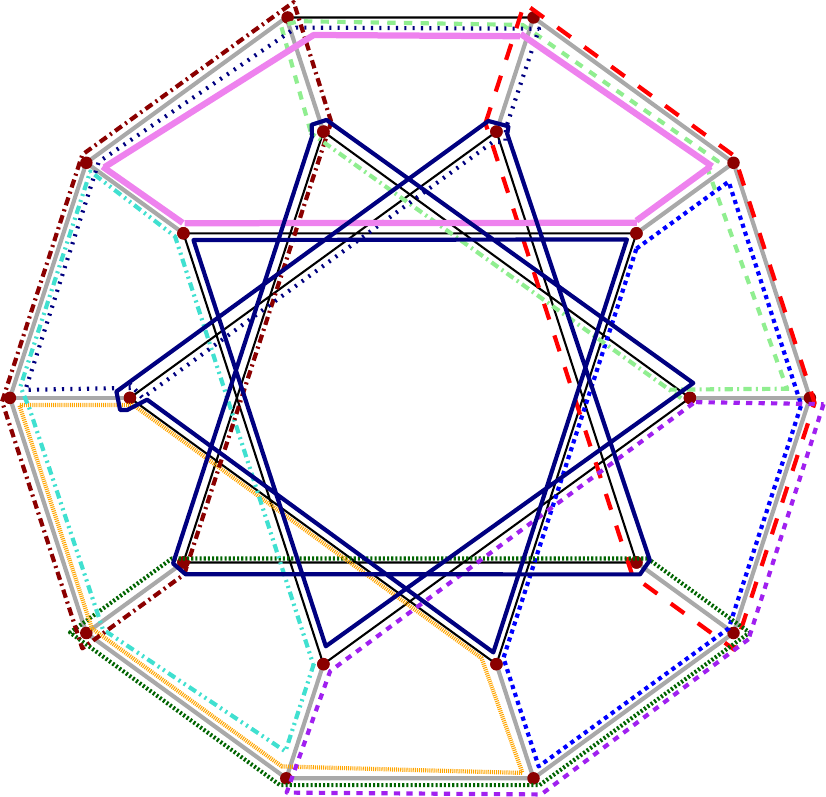}}}
\qquad
 \subfloat[Spanning tree]
    {{\includegraphics[scale=0.4,page=1]{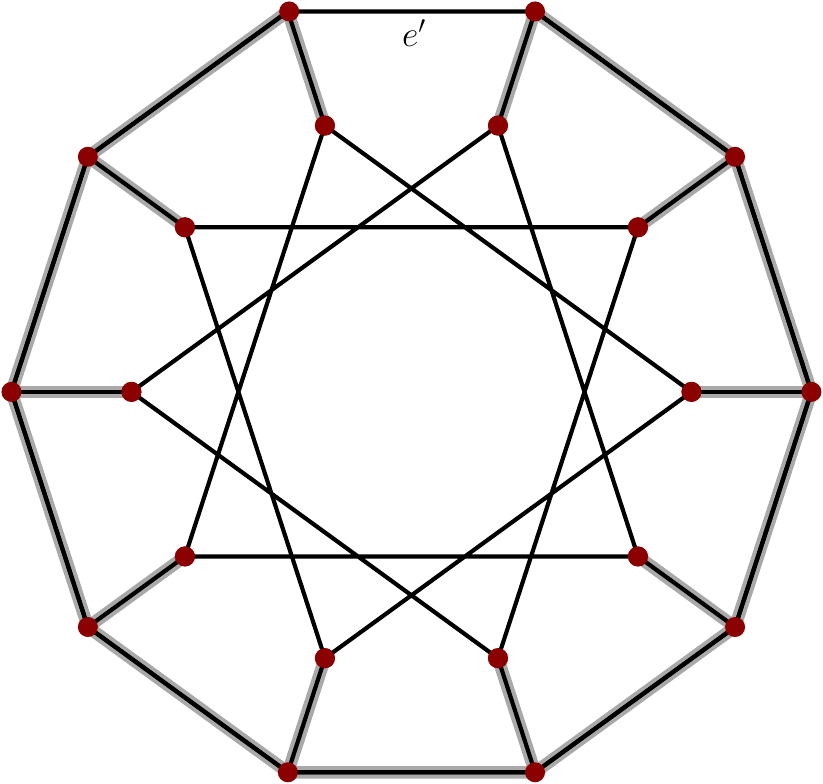}}}
\caption{ The Desargues graph with a 3-basis as indicated,  and the spanning tree used to argue that this is a basis.
        }
\label{fig:Desargues_basis}
\end{figure}

Let $ T $ be the spanning tree consisting of all outer edges except one edge $ e' $, along with all spike edges.
It suffices to argue that any fundamental cycle $C_f$ of $T$ can be generated by $\mathcal{B}$, where $f$ is an edge not in $T$.   If $f=e'$ then $C_f$ is the cycle $C_{\text{outer}}$ of outer edges; this can be generated by adding up $D_e$ for all inner edges $e$ and adding $D_{\text{inner}}$.  If $f$ is an inner edge, then either $C_f=D_f$ or $C_f=D_f+C_{\text{ outer}}$.   
So $C_f$ can always be generated and $\mathcal B$ forms a 3-basis.
\end{proof}

In the end, we retreated to the proof of \Cref{notbounded} to construct a relatively small 1-planar graph with basis number 4.

\begin{prop}\label{prop:sub_tutte_b_4}
There exists a 1-planar graph with 34 vertices and maximum degree 3 
that has basis number 4 or more.
\end{prop}
\begin{proof}
Consider the \defin{Tutte-Coxeter graph}, also known as the \defin{Tutte 8-cage} and shown in
\Cref{fig:no3basis2}. This graph is 3-regular with 30 vertices, and 
has \defin{girth} 8 (i.e., its shortest cycle has length 8).
An easy counting argument (see \cite{MR1844307}) 
can be used to show that a 3-regular graph with a 3-basis has girth
at most 7, so the Tutte-8-cage graph has basis number 4 or more.   
\begin{figure}[H]
\centering
\subfloat[Symmetric]
{{\includegraphics[scale=0.5,page=4,trim=0 140 0 100,clip]{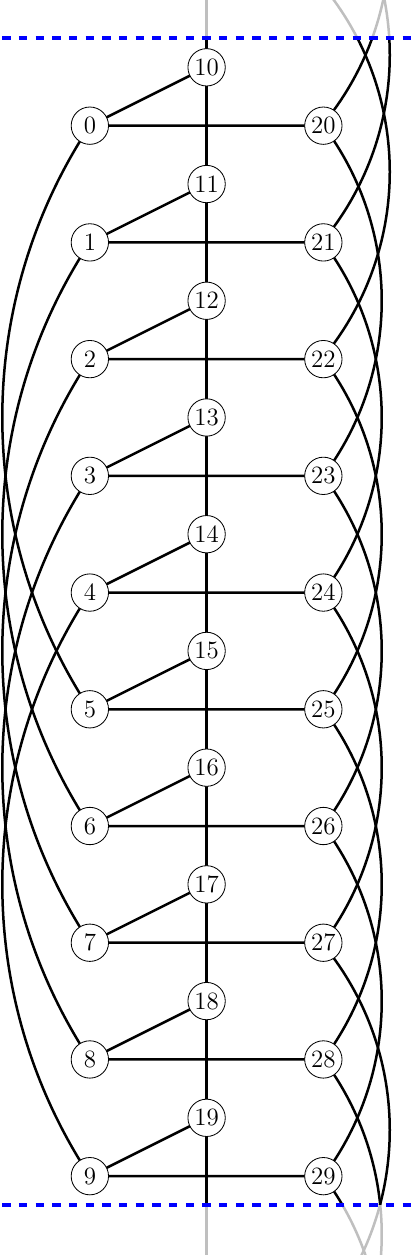}}}
\qquad
\subfloat[1-planar]
{{\includegraphics[scale=0.55,page=3]{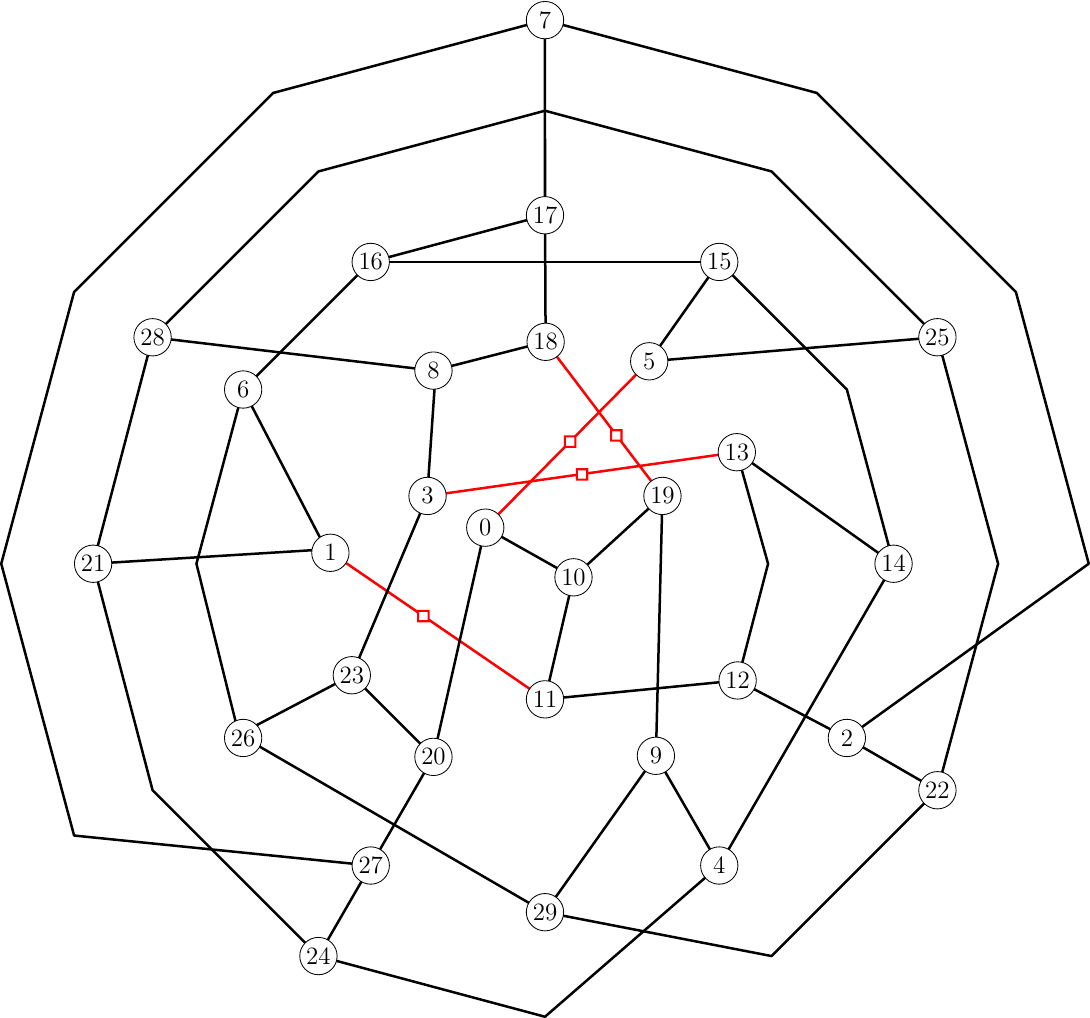}}}
\caption{ 
The Tutte 8-cage (we draw it on the rolling cylinder for ease of reading).
We can redraw this so that only four edges are crossed twice; subdividing these edges (white squares) gives a 1-planar graph.
   }
    \label{fig:no3basis2}
    \label{fig:Tutte8cage}
\end{figure}
It is not known whether the Tutte 8-cage is 1-planar, but we can apply the same trick
as in the proof of \Cref{notbounded} to turn it into a 1-planar graph by subdividing
edges.    \Cref{fig:no3basis2}(b) shows that if we start with a suitable drawing, then
it suffices to subdivide four edges, which gives us the desired 1-planar graph that has a basis
number at least 4 since the Tutte 8-cage graph can be obtained from it by contractions.
\end{proof}

While the constructed graph is 1-planar, it is not poppy 1-planar (and we do not know a way to make it poppy 1-planar that guarantees not to decrease the basis-number).   On the other hand, we proved that poppy 1-planar graphs have a 3-basis only 
under the restriction of a balanced orientation of the skirt walks.
Can this condition be dropped?

\begin{op}
Is there a poppy 1-planar graph $G$ with $b(G) = 4$?
\end{op}

The constructed graph is also not locally maximal.    We know that locally maximal 1-planar graphs have a 4-basis, but (unless they happen to be full-crossing 1-planar) it is not clear whether they have a 3-basis.

\begin{op}
    Is there a locally maximal 1-planar graph $G$ with $b(G)= 4$?
\end{op}

The most natural candidate graph for this open problem would be the graph from \Cref{skirt}(b), i.e., the graph obtained by taking multiple copies of $K_8-M$ and identifying them along one edge.   We have (for four or more copies) not been able to find a 3-basis for this graph, and we suspect that it has basis number 4, but this remains open.
\section{Conclusion}
\label{sec:conclusion}

In this paper, we studied the basis number of 1-planar graphs.  Even though these are ``close'' to being planar, their basis number can be unbounded (in contrast to planar graphs).    But if we impose additional restrictions (such as a connected skeleton) then their basis number is bounded, and under further restriction their basis number is three.

We close the paper with 
some open problems.
%
\begin{itemize}
    \item We achieved a 1-planar graph with a large basis number via the operation of subdividing edges, which clearly reduces the connectivity of the graph to at most 2.    Is it true that all 1-planar graphs with high connectivity have bounded basis number?   (The question appears open even for graphs that are not 1-planar.)

\item It is known that the basis number of graphs can be unbounded, but how big can it be relative to the number $n$ of vertices?    For example, does every graph have basis number $O(\log n)$?   (The question could be asked for graphs in general or specifically for 1-planar graphs).

Turning the question around, what is the smallest graph whose basis number is at least $k$?   This is open even for $k=4$, though the Tutte 8-cage shows that here the answer is at most 30 vertices.

\item What is the complexity of computing the basis number?   Is this NP-hard?   Is it NP-hard for 1-planar graphs?   How about testing whether the basis number of a 1-planar graph is 3?
\end{itemize}


\bibliographystyle{plainurlnat} 

\bibliography{ml1p}

\end{document}